\documentclass[12pt]{amsart}
\usepackage{amsmath, amsthm, amssymb, geometry, url, mathtools, mathrsfs}
\usepackage{arydshln}

\usepackage{etaremune}

\usepackage{xcolor}
\usepackage[colorlinks = true,
            linkcolor = blue,
            urlcolor  = blue,
            citecolor = blue,
            anchorcolor = blue]{hyperref}

\evensidemargin \oddsidemargin
\allowdisplaybreaks

\author{Sara Lapan}
\address{Department of Mathematics\\ University of California\\ Riverside, CA 92521}
\email{sara.lapan@ucr.edu}

\author{Benjamin Linowitz}
\address{Department of Mathematics\\ Oberlin College\\ Oberlin, OH 44074}
\email{benjamin.linowitz@oberlin.edu}

\author{Jeffrey S. Meyer}
\address{Department of Mathematics\\ 
California State University\\ 
San Bernardino, CA 92407}
\email{jeffrey.meyer@csusb.edu}

\title{Systole inequalities up congruence towers\\ for arithmetic locally symmetric spaces}

\usepackage{fancyhdr}

\fancypagestyle{plain}

\DeclareMathAlphabet{\curly}{U}{rsfs}{m}{n}

\DeclareMathOperator{\Ad}{Ad}

\DeclareMathOperator{\diag}{diag}

\DeclareMathOperator{\Mat}{M}

\DeclareMathOperator{\nr}{nr}

\DeclareMathOperator{\SL}{SL}
\DeclareMathOperator{\SO}{SO}
\DeclareMathOperator{\SU}{SU}

\DeclareMathOperator{\sys}{sys}

\DeclareMathOperator{\Sp}{Sp}

\DeclareMathOperator{\tr}{tr}

\DeclareMathOperator{\Vol}{Vol}

\DeclareMathOperator{\arccosh}{arccosh}

\DeclarePairedDelimiter\norm{\lVert}{\rVert}%

\newtheorem{thm}{Theorem}[section]
\newtheorem{cor}[thm]{Corollary}

\newtheorem{prop}[thm]{Proposition}
\newtheorem{lem}[thm]{Lemma}
\theoremstyle{definition}

\newtheorem{rem}[thm]{Remark}
\newtheorem*{rmk}{Remark}

\newtheorem*{ack}{Acknowledgements}
\theoremstyle{remark}

\newtheorem{proposition}{Proposition}[section]
\def\1{\mathbf{1}}

\theoremstyle{plain}
\newtheorem{mainthm}{Theorem}

\theoremstyle{remark}

\theoremstyle{plain}

\newtheorem{lemma}[proposition]{Lemma}

\def\C{\mathbf{C}}

\def\Q{\mathbf{Q}}

\def\R{\mathbf{R}}

\def\Q{\mathbf{Q}}
\def\Z{\mathbf{Z}}

\def\1{\mathbf{1}}

\newcommand{\abs}[1]{\left\vert#1\right\vert}

\makeatletter
\def\moverlay{\mathpalette\mov@rlay}
\def\mov@rlay#1#2{\leavevmode\vtop{%
   \baselineskip\z@skip \lineskiplimit-\maxdimen
   \ialign{\hfil$\m@th#1##$\hfil\cr#2\crcr}}}
\newcommand{\charfusion}[3][\mathord]{
    #1{\ifx#1\mathop\vphantom{#2}\fi
        \mathpalette\mov@rlay{#2\cr#3}
      }
    \ifx#1\mathop\expandafter\displaylimits\fi}
\makeatother

\makeatletter
\let\@@pmod\pmod
\DeclareRobustCommand{\pmod}{\@ifstar\@pmods\@@pmod}
\def\@pmods#1{\mkern4mu({\operator@font mod}\mkern 6mu#1)}
\makeatother

\begin{document}

\begin{abstract}
In this paper we study the systole growth of arithmetic locally symmetric spaces up congruence covers and show that this growth is at least logarithmic in volume. This generalizes previous work of Buser and Sarnak, and Katz, Schaps and Vishne in the context of compact arithmetic hyperbolic manifolds of dimension $2$ and $3$.\end{abstract}

\maketitle 


\vspace{-2pc}
\section{Introduction}

The \textit{systole} of a compact Riemannian manifold $M$ is the least length of a non-contractible loop on $M$. The systole $\sys(M)$ of $M$ and the volume $\Vol(M)$ of $M$ are deeply related and have been the focus of considerable research. For instance, Gromov \cite{gromov} showed that for a closed aspherical $n$-manifold $M$, there exists a constant $c:=c(n)$ such that
\begin{align}
\sys(M)\le c(\Vol(M))^{1/n}.
\end{align}
Note that all compact locally symmetric spaces of nonpositive curvature are closed aspherical.

Of particular interest has been the study of how systoles grow along congruence covers of a given base manifold $M$. Buser and Sarnak \cite{BS} showed that when $M$ is a compact arithmetic hyperbolic surface arising from a quaternion division algebra over $\Q$ there exists a constant $c:=c(M)$ such that the principal congruence covers $\{M_I\}$ of $M$ satisfy 
\begin{align}\label{eq:bs}
\sys(M_I) \ge \frac{4}{3}\log( \mathrm{g}(M_I))-c,
\end{align} 
where $\mathrm{g}(M_I)$ denotes the genus of $M_I$ and $\log$ denotes the natural logarithm. This result was later extended to principal congruence covers of arbitrary compact arithmetic hyperbolic surfaces by Katz, Schaps and Vishne \cite{KSV}. Furthermore, Katz, Schaps and Vishne proved an analogous result for compact arithmetic hyperbolic $3$-manifolds; namely, for a suitable constant $c:=c(M)$, the principal congruence covers $\{M_I\}$ of $M$ satisfy
\begin{align}\label{eq:ksv}
 \sys(M_I) \ge \frac{2}{3}\log (\norm{M_I})-c,
\end{align}
where $\norm{M_I}$ denotes the simplicial volume of $M_I$ (see \cite[Chapter 6]{Thurston}). These results were later generalized by Murillo to arithmetic hyperbolic manifolds of dimension $n$ which are of the first type \cite{M2}  and to Hilbert modular varieties \cite{M1} in the case that $I$ varies across the set of prime ideals in a certain number field.


The goal of this paper is to prove a generalization of the aforementioned Buser--Sarnak inequality (\ref{eq:bs}) for all arithmetic simple locally symmetric manifolds.

Unlike in the case of hyperbolic manifolds, there are multiple natural choices for how to scale the metric on a generic locally symmetric manifold.  
In Section \ref{section:prelims} we discuss such choices and explain in detail how scaling the metric affects the systole, volume, and systole growth up a tower of covers.
For instance, in Proposition \ref{prop:fundamental2} we show the important property that if systole growth is at least logarithmic up a tower relative to a given metric, then for any rescaling of the metric, systole growth is still at least logarithmic.

Before proving the general case, we focus on \textit{standard special linear manifolds}.  
A \textit{special linear manifold} of degree $n$ is a manifold of the form $M_\Gamma:=\Gamma\backslash \SL_n(\R) / \SO(n)$ where $\Gamma\subset \SL_n(\R)$ is a torsion-free lattice.
We call such lattices \textit{standard} when they arise from central simple algebras (see Section \ref{section:centralalgebralattices}), and we note that this terminology is in analogy to how arithmetic hyperbolic manifolds arising from quadratic forms are called \textit{standard arithmetic hyperbolic manifolds} \cite[4.10]{Meyer}.
Note that standard special linear manifolds are arithmetic and all arithmetic hyperbolic $2$-manifolds are also standard special linear manifolds of degree 2.

With Proposition \ref{prop:fundamental2} in mind, we choose a convenient normalization of the metric on special linear manifolds so that the sectional curvature is bounded between 0 and 1, and we call this normalization the \textit{geometric metric} $g$.
Given a standard special linear manifold $M$ and a rational prime $p$, we denote by $\{M_{p^m}\}$ the principal $p$-congruence tower of $M$ (see Section \ref{section:congruencesubgroups}).
We show that the systole growth up all but finitely many $p$-congruence towers is at least logarithmic in volume. 
Note that by defining the systole of a noncompact locally symmetric manifold to be the minimal length of a closed geodesic, we may consider systole growth along congruence towers of such noncompact manifolds as well.

\begin{mainthm}\label{thm:speciallinear}
Let $M$ be a standard special linear manifold of degree $n$, $n\ge 2$. There exists a constant $c:=c(M,g)$ such that for all but finitely many primes $p$ and all positive integers $m$,
\begin{align}\label{eq:generalbs}
 \sys(M_{p^m},g)\ge \frac{2\sqrt{2}}{n(n^2-1)} \log(\Vol(M_{p^m},g)) - c.
\end{align}
\end{mainthm}

In the special case when $n=2$ and $M$ is compact, the Gauss--Bonnet theorem states that the genus $\mathrm{g}(M_{p^m})$ of $M_{p^m}$ satisfies $\mathrm{g}(M_{p^m})=\frac{\Vol(M_{p^m},g)}{2\pi}+1$, and hence Theorem \ref{thm:speciallinear} gives
\begin{align*}
\sys(M_{p^m}) \ge \frac{\sqrt{2}}{3} \log(\mathrm{g}(M_{p^m})) - c'
\end{align*}
where $c'=c'(M)$ is a constant.  Observe that this is close to recovering \eqref{eq:bs}.

These methods enable use to prove a similar result for noncompact standard real, complex, or quaternionic arithmetic manifolds (see Section \ref{section:hyperbolic} for constructions).  
Unless otherwise stated, real hyperbolic manifolds will be given the hyperbolic metric in which they have constant sectional curvature $-1$, and the systole and volume will be scaled accordingly.  Similarly, we scale the metrics on complex and quaternionic hyperbolic manifolds so that their sectional curvature is bounded between $-1$ and $-\frac{1}{4}$.

%


\begin{mainthm}\label{thm:standardhyperbolic}
Let $N$ be a noncompact standard real (resp. complex, quaternionic) arithmetic hyperbolic manifold of dimension $n$ (resp. $2n$, $4n$), $n\ge 2$.
There exists a constant $c_2:=c_2(N)$  such that for all but finitely many primes $p$ and all positive integers $m$,
\begin{align}
 \sys(N_{p^m})\ge c_1 \log(\Vol(N_{p^m}))-c_2,
\end{align}
where 
$$c_1=\begin{cases}
\dfrac{2\sqrt{2}}{n(n+1)^2} & \mbox{ when $N$ is real hyperbolic,}\\
\dfrac{1}{ n(n+1)(n+2)} & \mbox{ when $N$ is complex hyperbolic,}\\
\dfrac{1}{2\sqrt{2}(n+1)^2(2n+3)} & \mbox{ when $N$ is quaternionic hyperbolic.}
\end{cases}$$
\end{mainthm}

Note that in \cite{M2} Murillo used different methods in order to obtain, in the case that $N$ is a standard real arithmetic hyperbolic $n$-manifold, a result analogous to Theorem \ref{thm:standardhyperbolic} with a constant of $\frac{8}{n(n+1)}$. Moreover, in an appendix to \cite{M2} D\'oria and Murillo show that the constant $\frac{8}{n(n+1)}$ is sharp. It is not known what the optimal constants are in the cases of complex and quaternionic hyperbolic manifolds.

A major ingredient in our proofs of Theorems \ref{thm:speciallinear} and  \ref{thm:standardhyperbolic}  is our Trace-Length Bounds Theorem \ref{thm:tracelengthbounds}.  In both $\SL_2(\R)$ and $\SL_2(\C)$, the translation length of a semisimple element can be understood in terms of the trace of the element.  This relationship has proven to be extremely useful, as the trace is well understood from a number theoretic perspective.  In $\SL_n(\R)$, $n\ge 3$, the relationship between translation length and trace is more nuanced.  Nevertheless, in our Trace-Length Bounds Theorem \ref{thm:tracelengthbounds} we prove upper and lower bounds for the translation length of a semisimple element in terms of the element's trace. 

In Section \ref{section:locallysymmetric} we show that each arithmetic simple locally symmetric manifold $N$ is commensurable to an immersed totally geodesic submanifold of a standard special linear manifold of explicitly bounded degree (Theorem \ref{thm:virtualimmersion}).  
Relative to this immersion, we endow each $N$ with the subspace metric which we also denote $g$.
For each rational prime $p$, this immersion induces a $p$-congruence tower $\{N_{p^m}\}$ above $N$.
While the induced $p$-congruence tower seems dependent upon the immersion, it is in fact natural in that it is commensurable of bounded distance (see Section \ref{section:prelims}) to the tower associated to the principal congruence subgroups $\mathrm{ker}(\mathrm{G}(\mathcal{O}_k)\to \mathrm{G}(\mathcal{O}_k/p^m\mathcal{O}_k))$ (see Remark \ref{rem:natrualcongruence}).

In addition to its associated Riemannian volume, each $N$ has an \textit{arithmetic measure} $\mu_a$ in the sense of Prasad \cite{P}.  
We believe that $\mu_a$ is the most natural measure for a general $N$, the most easily computable thanks to Prasad's volume formula \cite[Theorem 3.7]{P}, and hence that stating our results in terms of $\mu_a$ is most likely to be of use.  
That being said, there is an analogous statement for when arithmetic measure is replaced by metric volume.

\begin{mainthm}\label{thm:general}
Let $N$ be an arithmetic simple locally symmetric manifold of dimension $n$ and arithmetic measure $\mu_a(N)<v$.
Then for all but finitely many primes $p$ and all positive integers $m$,
\begin{align}\label{eq:generalbs}
 \sys(N_{p^m},g)\ge c_1 \log(\mu_a(N_{p^m})) - c_2,
\end{align}
where $c_1:=c_1(n,v)$ and $c_2:=c_2(N)$ are explicit constants.
\end{mainthm}



Simple locally symmetric manifolds that are neither real nor complex hyperbolic are arithmetic, and hence Theorem \ref{thm:general} applies.  
In the case of arithmetic hyperbolic manifolds with the hyperbolic metric, we  prove the following theorem which makes explicit the dependence of the multiplicative constant on the volume.


\begin{mainthm}\label{cor:mainhyperbolic}
Let $N$ be an arithmetic hyperbolic $n$-manifold with hyperbolic volume less than $V$. There exists an absolute, effectively computable constant $c_1:=c_1(n)>0$, and a constant $c_2:=c_2(N)$  such that for all but finitely many primes $p$ and all positive integers $m$,
\begin{align}\label{eq:generalbs}
 \sys(N_{p^m})\ge \frac{c_1}{(\log(V))^3} \log(\Vol(N_{p^m})) - c_2.
\end{align}
\end{mainthm}

In Section \ref{section:proofsbc}, we show how to explicitly compute the constants $c_1$ from the previous two theorems.  
Observe that the multiplicative constants in Theorems \ref{thm:speciallinear} and \ref{thm:standardhyperbolic} depend only on dimension, while the multiplicative constants in Theorems \ref{thm:general} and \ref{cor:mainhyperbolic} depend on dimension and volume.

%

%




\begin{rmk}
All of the results in this paper hold as well for simple locally symmetric \textit{orbifolds}. Note that in the context of locally symmetric orbifolds a closed geodesic is not defined to be locally length minimizing but rather to be a closed curve that lifts to a closed geodesic in a finite-sheeted manifold cover.
\end{rmk}
\begin{ack}The second and third authors were partially supported by the U.S. National Science Foundation grants DMS 1107452, 1107263, 1107367 ``RNMS: Geometric Structures and Representation Varieties'' (the GEAR Network). The second author is partially supported by a Simons Collaboration Grant.
\end{ack}

\section{Preliminaries on Metrics, Lengths, Volumes, and Towers}\label{section:prelims}

%
In this paper we assume some familiarity with Riemannian manifolds, Lie groups, Lie algebras, and symmetric spaces.  For a detailed reference on these topics, we refer the reader to \cite{H}.  We now record a few facts and establish terminology that we use throughout this paper and which enable us to discuss a general theory of systole growth up towers of covers.


%
%
Let $(M,g)$ be a finite volume Riemannian manifold and $c_1,c_2\in \R$ be constants such that
\begin{align}\label{eq:systolevolumebound}
\sys (M,  g)\ge c_1\log (\Vol(M,  g))-c_2.
\end{align}
Such a systole-volume bound behaves nicely when scaling the metric or lifting to covers:

\begin{lem}\label{lem:fundamental1}${}$
\begin{enumerate}
\item If $\alpha \in \R_{>0}$, then $\sys (M, \alpha g)\ge c'_1\log (\Vol(M, \alpha g))-c'_2$
where 
\begin{align*}
c'_1= \sqrt{\alpha}c_1 &&c'_2=\sqrt{\alpha}\left(c_2+\frac{c_1\dim M}{2}\log \alpha\right).
\end{align*}
\item If $M_I\to M$ is an $s$-sheeted cover, then $\sys (M_I)\ge c_1\log (\Vol(M_I))-c'_2.$
where 
\begin{align*}c_2'=c_2+c_1\log s.
\end{align*}
\end{enumerate}
\end{lem}

\begin{proof}
Scaling the metric scales the corresponding systole and volume (see \cite[Chapter 1]{docarmo}):
\begin{align}\label{eq:scaling}
\sys(M,\alpha g)=\sqrt{\alpha}\sys(M,g)&&
\Vol(M,\alpha g)=\alpha^{\frac{\dim M}{2}}\Vol(M,g).
\end{align}
An $s$-sheeted cover satisfies $\Vol(M_I)=s\Vol(M)$.
The results follow by substituting these values into \eqref{eq:systolevolumebound}.
\end{proof}
%
%

A \textit{tower of covers} $\{M_I\}_{I\in \mathcal{S}}$ of $M$ is a set of finite sheeted covers of $M$ indexed by a poset $\mathcal{S}$ such that if $I<J$, then $M_J$ covers $M_I$.  
If the systole of each manifold in the tower satisfies a logarithmic volume lower bound as in \eqref{eq:systolevolumebound} with the same constants $c_1$ and $c_2$, then we say the \textit{systole growth is at least logarithmic in volume} up the tower. In other words, the systole growth up a tower of covers $\{M_I\}_{I\in \mathcal{S}}$ of $M$ is at least logarithmic in volume if there are constants $c_1$ and $c_2$ which depend only on $M$ such that $\sys (M_I)\ge c_1\log (\Vol(M_I))-c_2$ for all manifolds $M_I$ in the tower.

Two locally symmetric spaces $M$ and $M'$ are \textit{commensurable} if they share a common finite sheeted cover.  Equivalently, if $\Gamma$ and $\Gamma'$ are their corresponding lattices, then $\Gamma\cap \Gamma'$ has finite index in $\Gamma$ and $\Gamma'$.  If $\{M_I\}_{I\in \mathcal{S}}$ is a tower of covers of $M$, then we define the associated \textit{induced tower} $\{M'_I\}_{I\in \mathcal{S}}$ of $M'$ where, for each $I\in\mathcal{S}$, $M'_I$ is the finite sheeted cover of $M_I$ of degree $\abs{M'_I:M_I}\le \abs{\Gamma:\Gamma\cap\Gamma'}$ associated to the lattice $\pi_1(M_I)\cap \Gamma'$. 

We define two towers $\{M_I\}_{I\in \mathcal{S}}$ and $\{M'_I\}_{I\in \mathcal{S}}$ to be \textit{commensurable}, if there exists a tower $\{M''_I\}_{I\in \mathcal{S}}$ where for each $I\in \mathcal{S}$, $M''_I$ is a common finite sheeted cover of $M_I$ and $M'_I$.  We define commensurable towers to be of \textit{bounded distance} if there exists an integer $s\ge1$ such that for each $I\in \mathcal{S}$, the covering maps $M''_I\to M_I$ and $M''_I\to M'_I$ are of no more than $s$ sheets.   Commensurable of bounded distance is an equivalence relation between towers.  Induced towers are of bounded distance.

It follows from Lemma \ref{lem:fundamental1} (ii) that if the systole growth up $\{M_I\}_{I\in \mathcal{S}}$ is at least logarithmic in volume, and $\{M'_I\}_{I\in \mathcal{S}}$ is commensurable of bounded distance to $\{M_I\}_{I\in \mathcal{S}}$, then the systole growth up $\{M'_I\}_{I\in \mathcal{S}}$ is also at least logarithmic in volume, and furthermore, they have the same multiplicative constant.  
We record these observations in the proposition below.

\begin{prop}\label{prop:fundamental2}
Let $M$ be a finite volume locally symmetric space and $\{M_I\}_{I\in \mathcal{S}}$ be a tower of covers of $M$.  
If systole growth is at least logarithmic in volume up the tower, then:
\begin{enumerate}
\item this property is independent of the scaling of the metric; 
\item the systole growth is at least logarithmic in volume up a commensurable tower of bounded distance; and
\item upon fixing the metric, the multiplicative constant $c_1$ is an invariant of the bounded distance commensurability class of $\{M_I\}_{I\in \mathcal{S}}$.
\end{enumerate}
\end{prop}

Sometimes it is more convenient or natural to work with the scaled measure $\mu_\beta=\beta \Vol$, $\beta\in \R_{>0}$, on $N$. 
For example, for compact hyperbolic $n$-manifolds, Gromov showed that there exists a constant $\beta:=\beta(n)$ such that  simplicial volume is $\beta$ times hyperbolic volume \cite[Theorem 6.2]{Thurston}.
In Section \ref{section:proofsbc}, we shall be considering the arithmetic measure of an arithmetic simple locally symmetric space, which is a scaling of the metric volume.
%
A direct computation shows that that if there exists constants $c_1$ and $c_2$ such that $\sys (N)\ge c_1\log (\Vol(N))-c_2$, then  $\sys (N)\ge c_1\log (\mu_\beta(N))-c_2'$ for $c'_2=(c_2-c_1\log\beta)$. 
In such a case, it follows that systole growth up a tower is at least logarithmic in metric volume if and only if systole growth is at least logarithmic in measure (c.f. \eqref{eq:ksv}).

%

The symmetric space $\SL_n(\R) / \SO(n)$ comes naturally equipped with two Riemannian metrics: (1) the \textit{Killing metric} determined by the Killing form $B( X,Y)=2n\tr(XY)$ on $\mathfrak{sl}_n(\R)$ and (2) the \textit{geometric metric}, in which the hyperbolic slices corresponding to the natural inclusions $\SL_2(\R)\to \SL_n(\R)$ have constant sectional curvature of $-1$.  
These metrics are constant multiples of one another.
Relative to the scaled Killing form $\alpha B$,  $\alpha\in \R_{>0}$, the curvature of a section determined by orthonormal vectors $X,Y\in \mathfrak{sl}_n(\R)$ is $K(X,Y)=\frac{2n}{\alpha}\tr([X,Y]^2)$ \cite[V.3.1]{H}.  It follows that the geometric metric is determined by 
\begin{align}\label{eq:geometricform}
\langle X,Y\rangle :=\frac{1}{n}B(X,Y)=2\tr(XY), \hspace{2pc}X,Y\in \mathfrak{sl}_n(\R).
\end{align}

Many papers on locally symmetric spaces use (for instance \cite{PR}) the Killing metric, however since we are interested in isometric immersions of locally symmetric spaces of smaller dimensions, it will be more convenient to normalize to the geometric metric.

Each $x\in \SL_n(\R)$ has Jordan decomposition $x=x_sx_u$ where $x_s$ is \textit{semisimple} and  $x_u$ is \textit{unipotent}.  
Semisimple elements in $\SL_n(\R)$ are diagonalizable (possibly over $\C$).
When $\Gamma$ is a cocompact lattice in $\SL_n(\R)$, the Godement Compactness Criterion implies that it only has semisimple elements \cite[Theorem 5.3.3]{Witte}.
Every semisimple element has a polar decomposition $x=x_hx_e$ where its \textit{hyperbolic part} $x_h$ has all positive real eigenvalues and its \textit{elliptic part} $x_e$ has eigenvalues that lie on the unit circle.
In particular, if $\{a_1,\ldots, a_n\}$ are the eigenvalues of $x$, then $\{\abs{a_1},\ldots, \abs{a_n}\}$ are the eigenvalues of $x_h$.
Unless stated otherwise, in what follows $x\in \SL_n(\R)$ will denote a semisimple element and $\{a_1, \ldots a_n\}$ its eigenvalues.

Each $x$ stabilizes and translates along a geodesic axis in $\SL_n(\R) / \SO(n)$.
Let $\ell(x)$ denote the \textit{translation length} of $x$ relative to the geometric metric on $\SL_n(\R) / \SO(n)$.
Closed geodesics in $\Gamma \backslash \SL_n(\R) / \SO(n)$ are in bijective correspondence with $\Gamma$-conjugacy classes of semisimple elements in $\Gamma$.  
The length of a closed geodesic associated to the class of $x$ is the translation length of $x$.

Let $\mathrm{A}\subset \SL_n(\R)$ denote the Lie subgroup of diagonal matrices with positive entries and let $\mathfrak{a}$ denote its Lie algebra. 
The map $\log:\mathrm{A}\to \mathfrak{a}\subset \mathfrak{sl}_n(\R)$, sending $(b_1,\ldots , b_n)\mapsto (\log(b_1), \ldots, \log(b_n))$ is an isomorphism.
Then $y=\diag(\abs{a_1},\ldots, \abs{a_n}))\in \mathrm{A}$ is $\SL_n(\R)$-conjugate to $x_h$. Let $Y=\log(y)$.
Using \eqref{eq:geometricform}, (c.f. \cite[Section 12.1]{MaR}, \cite[Prop. 8.5]{PR})
\begin{align}\label{eq:geometriclength}
\ell(x)=\ell(x_h)=\ell(y)=\sqrt{\langle Y,Y\rangle}=\sqrt{2\tr(Y^2)}= \sqrt{2\sum_{i=1}^n\left(\log \abs{a_i}\right)^2}.
\end{align}


\section{Trace-Length Bounds Theorem }

In this section we state and prove a fundamental relationship between the traces and translation lengths of elements $x\in \SL_n(\R)$.
These relationships are particularly valuable since they enable us to leverage number theoretic techniques to analyze traces, thereby giving us geometric data about lengths.  
The proof of the theorem uses a variety of analytic techniques.

\begin{thm}[Trace--Length Bounds]\label{thm:tracelengthbounds}
For $x\in \SL_n(\R)$ semisimple,
\begin{align}\label{hyperbolictracelengthbounds}
\sqrt{2}\arccosh\left(\frac{\tr (x_h)}{n} \right) \leq \ell(x) \le 
\sqrt{2n}\arccosh\left(\left(\frac{\tr (x_h)}{n}\right)^{n-1} \right).
 \end{align}
 

Furthermore, if $|\tr(x)|\geq1$, then
\begin{align}\label{tracelengthbounds}
\sqrt{2}\arccosh\left(\max\left\{1,\frac{\abs{\tr (x)}}{n}\right\} \right) \leq \ell(x) \le 
\sqrt{2n}\arccosh\left(\left(2\sum_{l=1}^{n} |\tr(x^l)|\right)^{n-1} \right).
 \end{align}

\end{thm}

\begin{rem}
When $n=2$, it is known that for $x$ hyperbolic with eigenvalues $a$ and $\frac{1}{a}$,
\begin{align}\label{tracelengthn2}
\ell(x) = 2 \log \abs{a} = 2 \arccosh \left(\frac{\abs{\tr (x)}}{2}\right).
\end{align}
Observe that \eqref{hyperbolictracelengthbounds} gives a tight upper bound.
For $n\ge 3$, no such direct equality is known.
\end{rem}

\begin{rem}In \eqref{tracelengthbounds} the condition that $\abs{\tr(x)}\geq 1$ is necessary to apply Proposition \ref{contfun} to get the upper bound on $\ell(x)$.\end{rem} 

The remainder of this section is dedicated to proving Theorem \ref{thm:tracelengthbounds}. The lower bounds follow from \eqref{eq:geometriclength} and 
Proposition \ref{ineq} (below), while Propositions \ref{ineq2} and \ref{contfun} (below) along with \eqref{eq:geometriclength} yield the upper bounds.
Recall that for $x\in \SL_n(\R)$ semisimple, its eigenvalues $a_1, \ldots a_n$ are complex numbers that satisfy $\sum_{i=1}^n a_i= \tr x$ and $\prod_{i=1}^n a_i= 1$.

\begin{lem}\label{jensen} For any $\displaystyle{\{a_i\}_{i=1}^n}\subset\C$ and $\beta\in\R$, if $\displaystyle\prod_{i=1}^n a_i=1$, then
$\displaystyle\sum_{i=1}^n |a_i|^\beta\geq n$.
\end{lem}
\begin{proof}
This is simply an application of the arithmetic and geometric means inequality:
$$\frac{1}{n}\sum_{i=1}^n \abs{a_i^\beta}\geq\left(\prod_{i=1}^n \abs{a_i^\beta}\right)^\frac{1}{n}=\abs{\prod_{i=1}^n a_i}^\frac{\beta}{n} =1.$$ 
\end{proof}

\begin{prop} \label{ineq}For any $\displaystyle\{a_i\}_{i=1}^n\subset\C$ satisfying $\displaystyle\prod_{i=1}^n a_i=1$, we have:
$$\arccosh\left(\max\left\{1,\frac{1}{n}\abs{\sum_{i=1}^n a_i}\right\}\right)\leq\arccosh\left(\frac{1}{n}\sum_{i=1}^n |a_i|\right)\leq \sqrt{\sum_ {i=1}^n\left(\log\abs{a_i}\right)^2}.$$
\end{prop}

Note that in Proposition \ref{ineq}, we can replace each $a_i$ with $a_i^{-1}$ in the two inequalities on the left and get the identical bound on the right.

\begin{proof} We will begin by proving an inequality for the sum inside of $\arccosh$.

\begin{align}
\frac{1}{n}\abs{\sum_{i=1}^n a_i}
&\leq
\frac{1}{n}\sum_{i=1}^n \abs{a_i}=
%
%
\frac{2}{n}\left(\sum_{i=1}^n \frac{1}{2}\left(e^{\log\abs{a_i}}+e^{-\log\abs{a_i}}\right)-\frac{1}{2}\sum_{i=1}^n\abs{a_i}^{-1}\right) \notag\\
&=\frac{2}{n}\left(\sum_{i=1}^n \cosh(\log\abs{a_i})-\frac{1}{2}\sum_{i=1}^n\abs{a_i}^{-1}\right)\notag \\ 
%
%
&=\frac{2}{n}\left(n+\sum_{m=1}^\infty\left(\sum_{i=1}^n \dfrac{(\log\abs{a_i})^{2m}}{(2m)!}\right)-\frac{1}{2}\sum_{i=1}^n\abs{a_i}^{-1}\right)\label{TS}\\
%
%
&\leq\frac{2}{n}\left(\frac{n}{2}+\sum_{m=1}^\infty\left(\dfrac{\Big(\sum_i (\log\abs{a_i})^2\Big)^{m}}{(2m)!}\right)\right) \label{lem} \\
&\leq\frac{2}{n}\left(\frac{n}{2}\sum_{m=0}^\infty\left(\dfrac{\left(\sqrt{\sum_i (\log\abs{a_i})^2}\right)^{2m}}{(2m)!}\right)\right) \notag\\
&=\cosh\left(\sqrt{\sum_{i=1}^n(\log\abs{a_i})^2} \right)  \label{cosh}  
\end{align}
Equations \eqref{TS} and \eqref{cosh} follow from the Taylor Series expansion of $\cosh(x)=\displaystyle\sum_{m=0}^\infty\dfrac{x^{2m}}{(2m)!}$, while equation \eqref{lem} makes use of Lemma \ref{jensen} to get $n-\frac{1}{2}\sum_i\abs{a_i}^{-1}\leq\frac{n}{2}$. 
Since $\arccosh(x)$ is increasing on $[1,\infty)$ and $\sum_i\abs{a_i}\geq n$ by Lemma \ref{jensen}, we have: 
$$\arccosh\left(\max\left\{1,\frac{1}{n}\abs{\sum_{i=1}^n a_i}\right\}\right)
\leq\arccosh\left(\frac{1}{n}\sum_{i=1}^n |a_i|\right)
\leq\sqrt{\sum_{i=1}^n(\log|a_i|)^2}.$$
\end{proof}

\begin{lemma}\label{betabound}For any $\displaystyle\{a_i\}_{i=1}^n\subset\C$ satisfying $\displaystyle\prod_{i=1}^n a_i=1$ and any $\beta>0$ we have: 
$$\cosh\left(\beta\sqrt{\frac{1}{n}\sum_{i=1}^n\left(\log|a_i|\right)^2}\right)\leq\frac{1}{2n}\sum_i  \left(|a_i|^\beta+|a_i^{-1}|^{\beta}\right)$$
\end{lemma}

\begin{proof}
\begin{align}
\cosh\left(\beta\sqrt{\frac{1}{n}\sum_{i=1}^n\left(\log|a_i|\right)^2}\right)
&= \sum_{m=0}^\infty\frac{1}{(2m)!}\left(\frac{1}{n}\sum_{i=1}^n\left(\beta\log|a_i|\right)^2\right)^{m} \label{TS2}\\
&\leq\sum_{m=0}^\infty\frac{1}{(2m)!}\left(\frac{1}{n}\sum_{i=1}^n\left(\beta\log|a_i|\right)^{2m}\right) \label{jen} \\
&=\frac{1}{n}\sum_{i=1}^n \cosh\left(\log|a_i|^{\alpha\sqrt{n^{-1}}}\right)\notag \\
&=\frac{1}{2n}\sum_{i=1}^n  \left(|a_i|^\beta+|a_i^{-1}|^{\beta}\right) \label{beta}
\end{align} 
In \eqref{TS2} we used the Taylor Series expansion of $\cosh(x)$. 
In \eqref{jen} we used that, for $m=0,1$ it is trivially true and for $m\geq 2$, $f(x)=x^{m}$ is convex when $x>0$ so we can apply Jensen's inequality: $f\left(\frac{\sum x_i}{n}\right)\leq\frac{\sum f(x_i)}{n}$.  
\end{proof}

\begin{prop} \label{ineq2}For any $\displaystyle\{a_i\}_{i=1}^n\subset\C$ satisfying $\displaystyle\prod_{i=1}^n a_i=1$ and any $\beta>0$ we have: 
 $$\sqrt{\sum_ {i=1}^n\left(\log\abs{a_i}\right)^2}\leq \frac{\sqrt{n}}{\beta}\arccosh\left( \left(\frac{1}{n}\sum_ {i=1}^n |a_i|^\beta\right)^{n-1}\right).$$
\end{prop}

\begin{proof}First we want to bound $\frac{1}{2n} \sum_i \left(|a_i|^\beta+|a_i^{-1}|^{\beta}\right)$ from Lemma \ref{betabound}.  Note that for any $i$, $ a_i^{-1}=\prod_{j\neq i} a_j$.  So $\sum_i a_i^{-1}=\sum_i \prod_{j\neq i} a_j$, which is an elementary symmetric polynomial that arises as the coefficient of the linear term in the characteristic polynomial $\prod (x-a_j)$. Similarly, $\sum_i |a_i|^{-\beta}=\sum_i \prod_{j\neq i} |a_j|^\beta$.  We can use this and Maclaurin's Inequality \cite{BC14} to bound $\sum_i|a_i|^{-\beta}$ above by $\sum_i|a_i|^\beta$.  In particular,
$$\frac{1}{\binom{n}{1}}\sum_{i=1}^n |a_i|^\beta\geq \left(\frac{1}{\binom{n}{n-1}}\sum_{i=1}^n \prod_{j\neq i} |a_j|^\beta\right)^\frac{1}{n-1}\geq \left(\prod_{i=1}^n |a_i|^\beta\right)^\frac{1}{n}=1.$$
Simplifying this,
$$\frac{1}{n}\sum_{i=1}^n |a_i|^\beta\geq \left(\frac{1}{n}\sum_{i=1}^n  |a_i|^{-\beta}\right)^\frac{1}{n-1}\geq 1$$
so
$$\frac{1}{n^{n-2}} \left(\sum_{i=1}^n  |a_i|^\beta\right)^{n-1} \geq  \sum_{i=1}^n  |a_i^{-1}|^{\beta} \geq n.$$

We use this to bound inequality \eqref{beta} as follows:
\begin{align*}
\frac{1}{2n} \sum_{i=1}^n  \left(|a_i|^\beta+|a_i^{-1}|^\beta\right)
&\leq 
\frac{1}{2n}  \left[\sum_{i=1}^n |a_i|^\beta+\frac{1}{n^{n-2}} \left(\sum_{i=1}^n  |a_i|^\beta\right)^{n-1} \right] \\
&\leq \frac{1}{2n}  \left[ \frac{2}{n^{n-2}} \left(\sum_{i=1}^n  |a_i|^\beta\right)^{n-1} \right] \\
&= \left(\frac{1}{n}\sum_{i=1}^n  |a_i|^\beta\right)^{n-1}, 
\end{align*}
where, in the last inequality, we use that $\frac{1}{n}\sum_i  |a_i|^\beta\geq 1$ by Lemma \ref{jensen} and so $\left(\frac{1}{n}\sum_i  |a_i|^\beta\right)^{n-1} \geq \frac{1}{n}\sum_i  |a_i|^\beta$.  Hence \eqref{TS2} is bounded above as follows:
$$\cosh\left(\beta\sqrt{\frac{1}{n}\sum_{i=1}^n\left(\log|a_i|\right)^2}\right)\leq  \left(\frac{1}{n}\sum_{i=1}^n |a_i|^\beta\right)^{n-1},$$

which we can re-write to get the desired result:
$$\sqrt{\sum_i\left(\log\abs{a_i}\right)^2}\leq \frac{\sqrt{n}}{\beta}\arccosh\left( \left(\frac{1}{n}\sum_i |a_i|^\beta\right)^{n-1}\right)$$
since $\arccosh(x)$ is increasing on $[1,\infty)$.  
\end{proof}

For $x\in \SL_n(\R)$, let 
\begin{align}\label{charpoly}
p_x(X)=X^n - s_1(x)X^{n-1}+ s_2(x)X^{n-2} -\cdots + (-1)^{n-1}s_{n-1}(x)X+(-1)^{n}
\end{align}
be the characteristic polynomial of $x$ where $s_j(x)$ denotes the $j^{th}$ elementary symmetric polynomial in the eigenvalues of $x$ (e.g. $\tr (x) = s_1(x)$ and $\det(x)=s_n(x)=1$).
Newton's identities  \cite{Kalman} state that for all $1\le j\le n$, 
\begin{align}\label{newton}
js_j(x) &=  s_{j-1}(x)\tr(x)-s_{j-2}(x)\tr(x^2)+\cdots +  (-1)^{j-2}s_{1}(x)\tr(x^{j-1}) + (-1)^{j-1}\tr(x^j)
%
\end{align}
In particular, these recursively show that each $s_j(x)$ can be written as a linear combination of $\{ \tr(x), \tr(x^2), \ldots \tr(x^j)\}$.  
Fujiwara's bound \cite{Fujiwara,Marden}, applied to our context states that if $\lambda$ is a root of the characteristic polynomial \eqref{charpoly},
then
\begin{align}\label{fujiwara}
\abs{\lambda}\le 2\max\left\{\abs{s_1(x)}, \abs{s_2(x)}^{\frac{1}{2}}, \ldots, \abs{s_{n-1}(x)}^{\frac{1}{n-1}}, 2^{-\frac{1}{n}}\right\}.
\end{align} 
Relationships \eqref{newton} and \eqref{fujiwara} are used in the proof of the following proposition.


\begin{prop}\label{contfun} Assume $\abs{\tr(x)}\geq1$. 
For each $n$, there exists a continuous function $F_n:\R^n\to \R$ such that 
$$\tr(x_h)\le F_n(\tr(x), \tr(x^2), \ldots, \tr(x^n)).$$
One such continuous function is: $F_n(z_1,\ldots,z_n)=2n\sum_{j=1}^n |z_j|.$
\end{prop}

\begin{proof} Claim: For $1\leq k\leq n-1$, $k|s_k(x)|\leq \left(\sum_{l=1}^k |\tr(x^l)|\right)^k$. \\ This is true for $k=1$ since $|s_1(x)|=|\tr(x)|$.  Suppose it is true for $1\leq k\leq j-1<n-1$.  Then, using Newton's Identities, for $k=j$:

\begin{align*}
j|s_j(x)|
& \leq |s_{j-1}(x)\tr(x)|+|s_{j-2}(x)\tr(x^2)|+\cdots +  |s_{1}(x)\tr(x^{j-1})| 
+ |\tr(x^j)| \\
& \leq \frac{|\tr(x)|}{j-1}\left(\sum_{l=1}^{j-1} |\tr(x^l)|\right)^{j-1}+
 \frac{|\tr(x^2)|}{j-2}\left(\sum_{l=1}^{j-2} |\tr(x^l)|\right)^{j-2}+\cdots+ |\tr(x^{j-1})\tr(x)|
 +|\tr(x^j)| \\
 &\leq \left(\sum_{l=1}^{j-1} |\tr(x^l)|\right)^{j-1}\left(\frac{|\tr(x)|}{j-1}+\frac{|\tr(x^2)|}{j-2}+\cdots+ |\tr(x^{j-1})|+|\tr(x^j)|\right) \\
 & \leq \left(\sum_{l=1}^{j} |\tr(x^l)|\right)^{j},
\end{align*}
where we use the assumption $\abs{\tr(x)}\geq1$ in the third inequality. This proves the claim.  
Combining this bound with \eqref{fujiwara} and using that $2^{-\frac{1}{n}}<1$ we get:
\begin{align*}
\tr(x_h)&\leq n\max_{1\leq i\leq n}|a_i|
\leq 2n \left(\sum_{l=1}^{n} |\tr(x^l)|\right) 
\end{align*}
\end{proof}

\section{Central Simple Algebras Over $\Q$ and Their Associated Orbifolds}\label{section:centralalgebralattices}

Let $A$ be a central simple algebra over $\Q$ of dimension $n^2\geq 4$. By Wedderburn's structure theorem there exists a positive integer $m$ and central division algebra $D$ over $\Q$ such that $A\cong \Mat_m(D)$. Therefore $n^2=m^2\dim_\Q(D)$.

Suppose now that $p$ is prime and consider the central simple algebra $A\otimes_\Q \Q_p \cong \Mat_m(D\otimes_\Q \Q_p)$ over $\Q_p$. This algebra also has dimension $n^2$ and, by Wedderburn's theorem, is isomorphic to $\Mat_{m_p}(D_p)$ for some positive integer $m_p$ and central division algebra $D_p$ over $\Q_p$. If the dimension of $D_p$ is greater than $1$ (equivalently, $m_p<n$) then we say that $p$ {\it ramifies} in $A$. Otherwise $p$ is {\it unramified} in $A$.

Let $K$ be an extension field of $\Q$ for which there is an isomorphism of $K$-algebras \[h: A\otimes_\Q K\rightarrow \Mat_n(K).\] Given an element $x\in A\otimes_\Q K$ the characteristic polynomial of $h(x)$ is well-defined and does not depend on the isomorphism $h$. For an element $a\in A$, the reduced characteristic polynomial of $a$ is defined as the characteristic polynomial of $h(a\otimes 1)$ and is of the form \[X^n - \tr(a)X^{n-1} +\cdots + (-1)^n\nr(a).\] We call $\tr(a)$ the {\it reduced trace} of $a$ and $\nr(a)$ the {\it reduced norm} of $a$.

We now define orders in central simple algebras over $\Q$. Let $A$ be a finite dimensional central simple algebra over $\Q$. A {\it $\Z$-order} $\mathcal O$ of $A$ is a subring of $A$ which is also a finitely generated $\Z$-submodule of $A$ for which $\mathcal O\otimes_\Z \Q\cong A$. An order of $A$ is {\it maximal} if it is not properly contained in any other order of $A$. A fundamental result \cite[Theorem 8.6]{Reiner} is that if $\mathcal O$ is an order of $A$ then the reduced characteristic polynomial of an element of $\mathcal O$ lies in $\Z[X]$. In particular if $x\in\mathcal O$ then both the reduced trace $\tr(x)$ and the reduced norm $\nr(x)$ of $x$ are integers.

We now discuss the construction of locally symmetric orbifolds from maximal orders in central simple algebras. Let $A$ be a central simple algebra of dimension $n^2$ over $\Q$ for which $A\otimes_\Q \R\cong \Mat_n(\R)$, and $\mathcal O$ be a maximal order of $A$. Denote by $\mathcal O^1$ the multiplicative subgroup of $\mathcal O^\times$ consisting of those elements with reduced norm one and by $\Gamma$ the image of $\mathcal O^1$ in $\SL_n(\R)$. Defined in this manner, $\Gamma$ is a lattice in $\SL_n(\R)$ with finite covolume \cite{BorelHarishChandra} (see also \cite{Witte} and the references therein). Let $M_\Gamma=\Gamma\backslash \SL_n(\R) / \SO(n)$ be the associated special linear orbifold. 
This orbifold is a manifold if and only if $\Gamma$ is torsion-free and is compact if and only if $A$ is a division algebra.
We call any orbifold commensurable with $M_\Gamma$ \textit{a standard special linear orbifold of degree $n$}.

\begin{rem}
Not every lattice in $\SL_n(\R)$ arises from the aforementioned construction. In particular there exist lattices in $\SL_n(\R)$ that are not commensurable with the ones coming from central simple algebras (see \cite{T1}).  Nevertheless we are able to restrict our attention to the lattices arising from central simple algebras because they are universal in the sense that all other lattices virtually embed into them in a controlled way. This will be described in Section \ref{section:locallysymmetric}.
\end{rem}
\section{Trace Estimates in Congruence Subgroups}\label{section:congruencesubgroups}

Let $A$ be a central simple algebra over $\Q$ of dimension $n^2\geq 4$, $\mathcal O$ be a maximal order of $A$ and $\Gamma$ be the lattice in $\SL_n(\R)$ associated to the elements of $\mathcal O^1$. Given a natural number $N\geq 1$ we have an ideal $N\mathcal O$ of $\mathcal O$ whose quotient $\mathcal O/N\mathcal O$ is a finite ring. We define the {\it level $N$ principal congruence subgroup} of $\mathcal O^1$ to be the kernel of the homomorphism $\mathcal O^1\rightarrow \left(\mathcal O/N\mathcal O\right)^\times$ obtained from the natural projection $\mathcal O \rightarrow \mathcal O/N\mathcal O$. We will denote this group by $\mathcal O^1(N)$; that is, $\mathcal O^1(N)=\ker\left(\mathcal O^1\rightarrow \left(\mathcal O/N\mathcal O\right)^\times\right)$. The image in $\SL_n(\R)$ of $\mathcal O^1(N)$ will be denoted $\Gamma(N)$. 

We now apply Theorem \ref{thm:tracelengthbounds} to examine the growth of traces of elements of $\Gamma$ in congruence subgroups.

\begin{thm}\label{thm:largetrace}
Let $\Gamma\subset \SL_n(\R)$ be the lattice defined above and let $p$ be a prime which does not ramify in $A$ and satisfies $p>2n$. 
For every $m\geq 1$ and semisimple $x\in \Gamma(p^m)$, $x\ne  1$, there is an integer $q$, $\abs{q}\le \frac{n}{2}$, such that $\abs{\tr(x^q)} > p^m-n$.
\end{thm}

\begin{proof}
Choose a basis of $A\otimes_\Q \Q_p$ so that $A\otimes_\Q \Q_p \cong \Mat_n(\Q_p)$ and $\mathcal O\otimes_{\Z} \Z_p \cong \Mat_n(\Z_p)$. Denote by $\varphi_p$ the natural projection $\varphi_p: \Mat_n(\Z_p)\rightarrow \Mat_n(\Z_p/p^m\Z_p)\cong \Mat_n(\Z/p^m\Z)$. Suppose that $x\in \Gamma(p^m)$. Identifying $x$ with its image in $\Mat_n(\Z_p)$, we have that $\varphi_p(x)=\mathrm{Id}_n$, hence $\tr(\varphi_p(x))=n$. Because $\mathcal O$ is a $\Z$-order of $A$ and $x\in \mathcal O$ we have that $\tr(x)\in\Z$, an observation which allows us to conclude that $\tr(x)\equiv n \pmod {p^m}$. This shows that if $x\in \Gamma(p^m)$ then $\tr(x)=p^mk+n$ for some $k\in\Z$.

Now suppose $x\in \Gamma(p^m)$ is semisimple  and $\tr(x^q)=n$ for each integer $q$, $\abs{q}\le \frac{n}{2}$.
Let $p_x(X)$
be the characteristic polynomial of $x$ as in \eqref{charpoly}.
Then by Newton's identities \eqref{newton}, and the fact that for $x\in \SL_n(\R)$  $s_i(x)=s_{n-i}(x^{-1})$, our assumptions on the traces of powers of $x$ uniquely determines each $s_j(x)$, and a computation shows that $p_x(X)=(X-1)^n$.  
Since $x$ is semisimple, we deduce that $x=1$.  
\end{proof}

\begin{cor}\label{cor:tracelengthcor}
Let $\Gamma$ be as above. For each semisimple $x\in \Gamma(p^m)$,  $x\ne  1$, 
\begin{align}\label{tracelengthbounds2}
\frac{2\sqrt{2}}{n}\arccosh\left(\frac{p^m-n}{n} \right) \leq \ell(x).
\end{align}
\end{cor}

\begin{proof}
Observe that $\ell(x)=\ell(x^{-1})$, $s_i(x)=s_{n-i}(x^{-1})$, and $\ell(x^q)=q\ell(x)$.  
This together with the Trace-Length Bounds Theorem \ref{thm:tracelengthbounds} gives the result.
\end{proof}

\begin{rem}As only finitely many primes ramify in a central simple algebra, Theorem \ref{thm:largetrace} hold for all but finitely many primes $p$.
\end{rem}
\section{Proof of Theorem \ref{thm:speciallinear}}

We now prove Theorem \ref{thm:speciallinear}. Let $A$ be a central simple algebra over $\Q$ of dimension $n^2\geq 4$, $\mathcal O$ be a maximal order of $A$ and $\Gamma$ be the image in $\SL_n(\R)$ of the multiplicative group of elements of $\mathcal O$ of reduced norm one. Let $M=\Gamma\backslash \SL_n(\R)/\SO(n)$. Given a prime $p$ and positive integer $m$ we denote by $M_{p^m}$ the congruence cover of $M$ of level $p^m$, and by $\abs{M_{p^m}:M}$ the degree of $M_{p^m}$ over $M$.

An immediate application of Corollary \ref{cor:tracelengthcor} is that 
\begin{align*}
\sys(M_{p^m})\geq \frac{2\sqrt{2}}{n}\arccosh\left(\frac{p^m-n}{n} \right),
\end{align*}

which implies that
\begin{equation}\label{firstsystolebound}
\sys(M_{p^m})\geq \frac{2\sqrt{2}}{n}\log\left(\frac{p^m-n}{n} \right)
\end{equation}

because $\arccosh z = \log (z + \sqrt{z^2-1})$.

Let $S$ be the set of rational primes which either ramify in $A$ or else satisfy $p<2n$.
Observe that $S$ is a finite set.
By construction, for each $p\notin S$, we have 
 \[\abs{M_{p^m}:M}=\abs{\Gamma:\Gamma(p^m)}\le\abs{\SL_n(\Z/p^m\Z)}\leq(p^m)^{n^2-1}.\]

Substituting this into (\ref{firstsystolebound}) and simplifying yields, for all $p\notin S$,
\begin{align*}
\sys(M_{p^m})&\geq \frac{2\sqrt{2}}{n}\left(\log\left(\abs{M_{p^m}:M}^{1/(n^2-1)} \right)-\log(2n)\right)\\
&= \frac{2\sqrt{2}}{n}\left(\frac{1}{n^2-1}\log(\Vol(M_{p^m})/\Vol(M))-\log(2n)\right)\\
&= \frac{2\sqrt{2}}{n(n^2-1)}\log(\Vol(M_{p^m}))-c,
\end{align*}
where $c$ is a positive constant depending on $M$. This proves Theorem \ref{thm:speciallinear} in the case that our special linear manifold $M$ is of the form $\Gamma\backslash \SL_n(\R)/\SO(n)$ with $\Gamma$ arising from the units of norm one in a maximal order of a central simple algebra of dimension $n^2$ over $\Q$. By Proposition \ref{prop:fundamental2}, the general case follows from this special case however, since by definition every standard special linear manifold is commensurable with one of the manifolds considered above.

\section{Simple Locally Symmetric Orbifolds: Immersions and Towers}\label{section:locallysymmetric}

In the remainder of the paper, we assume familiarity with algebraic and arithmetic groups.
For detailed references on these topics, we refer the reader to \cite{Borel, Witte}.
Let $k$ denote a totally real number field, $\mathcal{O}_k$ its ring of integers, and let $\mathrm{G}$ be a connected, simple, semisimple, adjoint algebraic $k$-group that is anisotropic at all but one real place of $k$.
Fix once and for all the infinite embedding $k\subset \R$ for which $\mathrm{G}$ is isotropic.  
Then $\mathrm{G}(\R)$ is a simple Lie group (in the sense that the complexification of its Lie algebra is simple) and $\mathrm{G}(\mathcal{O}_k)\subset \mathrm{G}(\R)$ is a \textit{principal arithmetic} lattice.  
Let $\mathrm{K}\subset \mathrm{G}(\R)$ be a maximal compact subgroup.  
Then $\mathrm{G}(\mathcal{O}_k)\backslash \mathrm{G}(\R) / \mathrm{K}$ is a \textit{principal arithmetic simple locally symmetric orbifold}.

A \textit{simple locally symmetric orbifold} is a Riemannian orbifold of the form $\Gamma\backslash \mathcal{G} / \mathcal{K}$ where $\mathcal{G}$ is a connected simple Lie group, $\mathcal{K}$ is a maximal compact, and $\Gamma$ is a lattice.  
Examples of simple locally symmetric orbifolds include finite volume real, complex, and quaternionic hyperbolic manifolds, as well as special linear manifolds.  
Such orbifolds are \textit{arithmetic} when $M$ is commensurable with a principle arithmetic simple locally symmetric orbifold.  By the work of Margulis \cite{Mar} and Gromov--Schoen \cite{GS}, all simple locally symmetric orbifolds other than real and complex hyperbolic are arithmetic.  


\begin{thm}\label{thm:virtualimmersion}
Let $N=\Lambda\backslash \mathrm{G}(\R)/\mathrm{K}$ be an arithmetic simple locally symmetric orbifold and $k$ be its field of definition.  Then there exists an orbifold commensurable to $N$ that can be immersed as a totally geodesic suborbifold of a standard special linear orbifold of degree $[k:\Q] \cdot \dim \mathrm{G}$.  
\end{thm}

Theorem \ref{thm:virtualimmersion} is a consequence of the following algebraic result.



\begin{prop}\label{prop:speciallinearembeddings}
Let $d_1=\dim \mathrm{G}$ and  $d_2=[k:\Q]$.
Then there exists a Lie group embedding $\rho:\mathrm{G}(\R)\to \SL_{d_1d_2}(\R)$ such that $\rho^{-1}(\rho(\mathrm{G}(\R))\cap \SL_{d_1d_2}(\Z))$ is commensurable with $\Lambda$.  
\end{prop}

\begin{proof}
Let $\mathfrak{g}$ denote the Lie algebra of $\mathrm{G}$. 
The adjoint representation $\Ad: \mathrm{G}\to \SL(\mathfrak{g})\cong \SL_{d_1}$ is $k$-rational \cite[3.13]{Borel} and since $\mathrm{G}$ is adjoint, it is injective.
Via restriction of scalars and the regular representation \cite[2.1.2]{PlatRap}, we get a sequence of $\Q$-rational injections:
$$\mathrm{R}_{k/\Q}(\mathrm{G})\to \mathrm{R}_{k/\Q}(\SL_{d_1})\to \SL_{d_1d_2}$$
Since $\mathrm{G}(\R)$ is naturally identified as the unique noncompact factor of  $\mathrm{R}_{k/\Q}(\mathrm{G})(\R)$, we obtain a sequence of Lie group injections:
$$\mathrm{G}(\R) \to (\mathrm{R}_{k/\Q}(\mathrm{G}))(\R) \to \SL_{d_1d_2}(\R)$$
whose composition we denote $\rho$.  
Choosing an $\mathcal{O}_k$-lattice $L\subset \mathfrak{g}$ determines a $\Z$-structure on $\SL_{d_1d_2}$ for which $\rho^{-1}(\rho(\mathrm{G}(\R))\cap \SL_{d_1d_2}(\Z))$ is commensurable with $\Lambda$ \cite[Prop. 6.2]{BorelHarishChandra}.
\end{proof}

We record $d_1=\dim \mathrm{G}$ as a function of absolute rank $r$ for each Killing--Cartan type below.
\begin{center}
\begin{tabular}{ r|*{9}{c}}
  \hline			
		& $A_r$ 	  & $B_r$ 	& $C_r$ 	&$D_r$	&$E_6$	&$ E_7$	& $E_8$	& $F_4$	& $G_2$\\
$d_1$	& $r^2+2r $&$2r^2+r$ 	& $2r^2+r$ 	&$2r^2-r$	& $78$	&$133$	&$248$	&$52$	&$14$\\
  \hline  
\end{tabular}
\end{center}

%



For each prime $p$, the $p$-congruence tower $\{M_{p^m}\}$ of $M=\SL_{d_1d_2}(\Z)\backslash \SL_{d_1d_2}(\R)/\SO(d_1d_2)$ induces a \textit{$p$-congruence tower} $\{N_{p^m}\}$ of $N$.

\begin{rem}\label{rem:natrualcongruence}
In Proposition \ref{prop:speciallinearembeddings}, the $\Q$-structure on $\SL_{d_1d_2}$ is canonical, but the $\Z$-structure requires a choice of an $\mathcal{O}_k$-lattice.
A different choice of $L$ would result in a commensurable cover \cite[Cor. 6.3]{BorelHarishChandra}.  
The two resulting induced towers over $N$ towers are commensurable of bounded distance, and hence by Proposition \ref{prop:fundamental2} the growth up an associated $p$-congruence tower is independent of the choice of $\mathcal{O}_k$-lattice.
Furthermore, since $(\mathrm{R}_{k/\Q}(\mathrm{G}))(\Z)\cong \mathrm{G}(\mathcal{O}_k)$ \cite[2.1.2]{PlatRap}, a $p$-congruence tower is of bounded distance from the tower associated to the principal congruence subgroups $\mathrm{ker}(\mathrm{G}(\mathcal{O}_k)\to \mathrm{G}(\mathcal{O}_k/p^m\mathcal{O}_k))$.
In this sense, up to bounded distance, $\{N_{p^m}\}$ is a natural tower only depending on $\mathrm{G}$ and $p$.
\end{rem}

\begin{rem}\label{rem:subspacekilling}
Endowing $N$  and its covers with the \textit{subspace metric}, $g$, enables us to use Theorem \ref{thm:speciallinear} to obtain coarse estimates for the growth of the systoles up the $\{N_{p^m}\}$.  
Observe that by the definition of the canonical embedding in Proposition \ref{prop:speciallinearembeddings}, the subspace metric is $2d_2$ times the Killing metric on $N$.
\end{rem}

\begin{cor}\label{cor:general_inequality_d1d2_vol}
Let $N$ be an arithmetic simple locally symmetric orbifold,  $k$ its field of definition and $\mathrm{G}$ its associated group.  
Let $d_1=\dim \mathrm{G}$ and  $d_2=[k:\Q]$.
There exists an explicit constant $c_1:=c_1(d_1, d_2)$ and constant $c_2:=c_2(N,g)$ such that for all primes $p>2d_1d_2$ and all positive integers $m$,
\begin{align}
 \sys(N_{p^m},g)\ge \frac{2\sqrt{2}}{d_1d_2(d_1^2d_2^2-1)} \log(\Vol(N_{p^m},g)) - c_2.
\end{align}
\end{cor}

In light if Proposition \ref{prop:fundamental2}, for every $N$, and sufficiently large $p$, systole growth is at least logarithmic in volume up every congruence $p$-tower, and this fact is independent of the choice of metric.
%

%

\section{Arithmetic Hyperbolic Orbifolds}\label{section:hyperbolic}

In this section, we specialize our results to real, complex, and quaternionic hyperbolic orbifolds and then prove Theorem \ref{thm:standardhyperbolic}.
Real hyperbolic $n$-orbifolds arise from lattices in $\SO_0(n,1)$.
Relative to the canonical embedding, $\mathfrak{so}(n,1)$ is a Lie subalgebra of $\mathfrak{sl}_{n+1}(\R)$, and the tangent space of hyperbolic $n$-space $\SO_0(n,1)/\SO(n)$ can be identified with
\begin{align}\label{eq:slnliealgebra}
\left\{ \left(
    \begin{array}{c;{2pt/2pt}r}
    \mbox{\LARGE $0$} & \begin{matrix} x_1\\ x_2\\ \vdots \\ x_n \end{matrix} \\ \hdashline[2pt/2pt]
    \begin{matrix} x_1& x_2& \ldots & x_n \end{matrix} & 0
    \end{array}
    \right) \ \Bigg| \ x_1, \ldots, x_n\in \R
    \right\}.
\end{align}
The Killing form on $\mathfrak{so}(n,1)$ is $(n-1)\tr(XY)$ and a direct computation \cite[V.3.1]{H} shows that the Killing metric has sectional curvature $-\frac{1}{2(n-1)}$.
The following is then a consequence of Theorem \ref{thm:virtualimmersion} and Remark \ref{rem:subspacekilling}.

\begin{cor}\label{cor:degree}
Let $N$ be an arithmetic real hyperbolic $n$-orbifold with field of definition $k$ of degree $d$.
Then $N$ is commensurable to an immersed totally geodesic subspace of a degree $d(2n^2+5n+3)$ special linear orbifold.
With respect to the subspace metric, this subspace has constant sectional curvature $-\frac{1}{4(n-1)d}$.
\end{cor}

Before stating the following corollary, we recall that real hyperbolic orbifolds will be given the hyperbolic metric $h_{\R}$ in which they have constant sectional curvature $-1$ unless an alternative Riemannian metric $g$ is explicitly given.

\begin{cor}\label{cor:hyperbolic_inequality_d1d2_vol}
Let $N$ be an arithmetic real hyperbolic $n$-orbifold, $n\ge 4$, with field of definition $k$ of degree $d$.
There exists a constant $c_2:=c_2(N)$  such that for all primes $p>2(2n^2+5n+3)d$ and all positive integers $m$,
\begin{align}
 \sys(N_{p^m})\ge \frac{\sqrt{2}}{144d^{7/2}n^{7/2}} \log(\Vol(N_{p^m})) - c_2.
\end{align}
\end{cor}

\begin{proof}
Let $g$ denote the subspace metric on $N$ and its covers.  By  scaling conversions \eqref{eq:scaling} and Theorem \ref{thm:speciallinear},
\begin{align*}
 \sys(N_{p^m}) 	&= \sqrt{\frac{1}{4(n-1)d}}\sys(N_{p^m},g)\\
 			&\ge \sqrt{\frac{1}{4(n-1)d}} \left(\frac{2\sqrt{2}}{((2n^2+5n+3)d)((2n^2+5n+3)d)^2-1)}\right) \log(\Vol(N_{p^m}),g) - c'_2\\
			&\ge \left(\frac{\sqrt{2}}{d^{7/2}n(2n^2+5n+3)^3}\right) \log(\Vol(N_{p^m}),g) - c'_2\\
			&\ge \frac{\sqrt{2}}{144d^{7/2}n^{7/2}} \log(\Vol(N_{p^m})) - c_2.
\end{align*}
\end{proof}

Complex and quaternionic hyperbolic space (which we denote $\mathbb{H}^n_{\C}$ and $\mathbb{H}^n_{\mathbf{H}}$, respectively) arise as the globally symmetric spaces associated to the Lie groups $\SU(n,1)$ and $\Sp(n,1)$, respectively.  The metric is often chosen so that the sectional curvature is pinched between $-1$ and $-\frac{1}{4}$.
We do so here and refer to this as their hyperbolic metrics, $h_\C$ and $h_{\mathbf{H}}$, respectively.  
In both of these cases, $\SO(n,1)$ naturally embeds as a Lie subgroup and, relative to their hyperbolic metrics, the associated totally geodesic, real hyperbolic space $\mathbb{H}^n_{\R}$ has constant sectional curvature $-\frac{1}{4}$ (see, for example, \cite{Epstein} and \cite{KPan}).  
We use this fact to renormalize from the subspace metric to the hyperbolic metric in Lemma \ref{lem:metricrenormalization} below.

Noncompact standard arithmetic lattices in each of these groups are constructed as follows.

\textit{Case 1: Real Hyperbolic.}  Let $(V,q)$ be a quadratic space over $\Q$ with signature $(n,1)$.  Let $G=\SO(V,q)$.  The canonical $\Q$-embedding $\mathrm{G}\to  \SL_{n+1}(\Q)$ is $\Q$-rational.

\textit{Case 2: Complex Hyperbolic.}  Let $\Q(\sqrt{d})/\Q$ be an imaginary quadratic extension with involution $\overline{x+y\sqrt{d}}=x-y\sqrt{d}$.  Let $(V,h)$ be a Hermitian space over $\Q(\sqrt{d})/\Q$ with signature $(n,1)$.  Let $\mathrm{G}=\SU(V,h)$.  Corresponding to the embedding 
\begin{align}\label{eq:fieldembed}
\Q(\sqrt{d})\to \mathrm{Mat}_2(\Q), \hspace{2pc} x+ y \sqrt{d} \mapsto \begin{pmatrix} x & d y \\ y & x \end{pmatrix},
\end{align}
the canonical $\Q$-embedding $G\to \SL_{n+1}(\Q(\sqrt{d}))\to \SL_{2n+2}(\Q)$ is $\Q$-rational.

\textit{Case 3: Quaternionic Hyperbolic.}  Let $D/\Q$ be a quaternion division algebra that splits over $\R$, with Hilbert symbol $\left(\frac{d,e}{\Q}\right)$ and $j\in D$ such that $j^2=e$.  Let $(V,h)$ be a Hermitian space over $D/\Q$ with signature $(n,1)$.  
Let $\mathrm{G}=\SU(V,h)$.  
Corresponding to \eqref{eq:fieldembed} and the embedding 
\begin{align}\label{eq:divembed}
D\to \mathrm{Mat}_2(\Q(\sqrt{d})), \hspace{2pc} w+zj \mapsto \begin{pmatrix} w & ez \\  \overline{z} & \overline{w} \end{pmatrix},
\end{align}
the canonical $\Q$-embedding $G\to \SL_{n+1}(D)\to \SL_{4n+4}(\Q)$ is $\Q$-rational.

In each of these cases, use the $\Q$-rational embedding to define $\Lambda$, the $\Z$-points of $\mathrm{G}$, and let $\mathrm{K}$ be the maximal compact subgroup of $\mathrm{G}(\R)$.  A space commensurable with $N=\Lambda\backslash \mathrm{G}(\R)/\mathrm{K}$ is a \textit{noncompact standard real (resp. complex, quaternionic) hyperbolic orbifold}.  

Let $X_{m}:=\SL_m(\R)/\SO(m)$ denote the degree $m$ special linear space and let $g$ denote its geometric metric \eqref{eq:scaling}.  Relative to the above constructions, $\mathbb{H}^n_{\R}$  (resp. $\mathbb{H}^n_{\C}$, $\mathbb{H}^n_{\mathbf{H}}$) embeds as a totally geodesic subspace of $X_{n+1}$ (resp. $X_{2n+2}$, $X_{4n+4}$) and let $g_{\R}$ (resp. $g_{\C}$, $g_{\mathbf{H}}$) denote the induced subspace metric.

\begin{lem}\label{lem:metricrenormalization} With the above notation, 
\begin{align}
h_{\R}=\frac{1}{4}g_{\R}, \hspace{2pc}
h_{\C}=\frac{1}{2}g_{\C}, \hspace{2pc}
h_{\mathbf{H}}=\frac{1}{4}g_{\mathbf{H}}.
\end{align}
\end{lem}

\begin{proof}
A direct computation shows that the sectional curvature \cite[V.3.1]{H} $(\mathbb{H}^n_{\R},g_{\R})$ is $-\frac{1}{4}$, and hence the metric of constant sectional curvature $-1$ is $\frac{1}{4}g$.
It follows from \eqref{eq:geometricform}, \eqref{eq:fieldembed}, and \eqref{eq:divembed} that the subspace metric on $\mathfrak{so}(n,1)$ in $\mathfrak{sl}_{2n+2}(\R)$ (resp. $\mathfrak{sl}_{4n+r}(\R)$) is twice (resp. four times) the subspace metric coming from $\mathfrak{sl}_{n+1}(\R)$, which has constant sectional curvature $-\frac{1}{4}$.  Hence $(\mathbb{H}_{\R}^n, g_{\C})$ (resp. $(\mathbb{H}_{\R}^n, g_{\mathbf{H}})$) has constant sectional curvature $-\frac{1}{8}$ (resp. $-\frac{1}{16}$).
Renormalizing the metric so that $\mathbb{H}_\R^n$ has constant sectional curvature $-\frac{1}{4}$ yields the desired results.
\end{proof}

\begin{proof}[Proof of Theorem \ref{thm:standardhyperbolic}]  
Begin by supposing that $N=\Lambda\backslash \mathrm{G}(\R)/\mathrm{K}$ as above.
It follows that $N_{p^m}= \Lambda(p^m)\backslash \mathrm{G}(\R)/\mathrm{K}$ where $\Lambda(p^m):=\Gamma(p^m)\cap \mathrm{G}(\R)$.
Without loss of generality, we may assume the form is given by $\langle a_1,\ldots a_n, -a_{n+1}\rangle$ where each $a_i$ is a positive integer.  Let $S$ be the set of odd primes that excludes the finitely many which divide the $a_i$. 
For each $p\in S$,   
%
$$\abs{N_{p^m}:N}=\abs{\Lambda:\Lambda(p^m)}\le\abs{\mathrm{G}(\Z/p^m\Z)}\leq(p^m)^{\dim \mathrm{G}}.$$

With respect to the subspace metric $g$ \eqref{eq:geometricform}, $(N_{p^m},g)$ is an isometrically immersed totally geodesic subspace of $M_{p^m}$.  Thus $\sys(N_{p^m}, g)\ge \sys(M_{p^m},g)$.
Let $b=1$, $2$, or $4$, depending upon whether $N$ is real, complex, or quaternionic.
Using the systole bound \eqref{firstsystolebound}:

\begin{align*}
\sys(N_{p^m},g)&\geq \frac{2\sqrt{2}}{b(n+1)}\left(\log\left(\abs{N_{p^m}:N}^{1/\dim \mathrm{G}} \right)-\log(2b(n+1))\right)\\
&= \frac{2\sqrt{2}}{b(n+1)}\left(\frac{1}{\dim \mathrm{G}}\log(\Vol(N_{p^m})/\Vol(N))-\log(2b(n+1))\right)\\
&= \frac{2\sqrt{2}}{b(n+1)\dim \mathrm{G}}\log(\Vol(N_{p^m}))-c.
\end{align*}

Note that $\dim \SO(n+1) = \frac{n(n+1)}{2}$, $\dim \SU(n+1)= n(n+2)$, and $\dim \Sp_{2n+2} = (n+1)(2n+3)$.  

By Lemma \ref{lem:metricrenormalization} and  \eqref{eq:scaling},
$\sys(N_{p^m})=\frac{1}{\sqrt{\kappa}}\sys(N_{p^m},g)$, where $\kappa=4$, $2$, or $4$, depending upon whether $N$ is real, complex, or quaternionic, which proves the growth bounds for principle standard arithmetic hyperbolic orbifolds.  
The general case follows from Proposition \ref{prop:fundamental2}.
\end{proof}

\section{Arithmetic Measure and the Proofs of Theorem \ref{thm:general} and Theorem \ref{cor:mainhyperbolic}}\label{section:proofsbc}

In the previous sections we established that for all but finitely many primes $p$, the systole growth up $p$-congruence covers is at least logarithmic in metric volume, and furthermore, we explicitly computed the multiplicative constant $c_1$ in terms of the algebraic data $\dim \mathrm{G}$ and $[k:\Q]$.  
In this section, we show how this algebraic data can be replaced with the geometric data of dimension and volume of $N$.
It is a direct computation for each $\mathrm{G}(\R)$ to write $\dim \mathrm{G}$ as a function of $\dim N$, see for example \cite[Ch. X. \S6. Table V]{H}.
Furthermore, the propositions of this section bound $[k:\Q]$ by an explicit function volume.
 Theorems \ref{thm:general} and \ref{cor:mainhyperbolic} then follow from Propositions \ref{prop:volumedegreebound} and \ref{prop:degreebounds}, respectively.  The propositions are highly technical and we assume familiarity with Prasad's volume formula \cite{P}.  


Let $N$, $\mathrm{G}$, $k$, and $r$ be as in the previous section and let $\widetilde{\mathrm{G}}$ denote the simply connected cover of $\mathrm{G}$.
If $\iota:\widetilde{\mathrm{G}}\to \mathrm{G}$ is the central isogeny, then we may lift the lattice $\pi_1(N)\subset \mathrm{G}(\R)$ to a lattice $\Lambda_N:=\iota^{-1}(\pi_1(N))\subset \widetilde{\mathrm{G}}(\R)$.  
It follows that $\Lambda_N \backslash \widetilde{\mathrm{G}}(\R)\cong   \pi_1(N) \backslash\mathrm{G}(\R)$.
By the \textit{arithmetic measure} $\mu_a$ on $N$ we mean the pushforward of Prasad's normalized Haar measure $\mu_\infty$ on $\Lambda_N \backslash \widetilde{\mathrm{G}}(\R)$ \cite[3.11]{P}.
In particular, $\mu_a(N):=\mu_\infty(\Lambda_N \backslash \widetilde{\mathrm{G}}(\R))$.

\begin{rem}\label{rem:volumemeasure}
As $\frac{\Vol(N_{p^m},g)}{\Vol(N,g)}=\frac{\mu_a(N_{p^m})}{\mu_a(N)}$, it follows that  $\log (\Vol(N_{p^m},g))=\log(\mu_a(N_{p^m}))+C$ where $C=C(N)$.
As such, Theorems \ref{thm:speciallinear} and \ref{thm:standardhyperbolic} and Corollaries \ref{cor:general_inequality_d1d2_vol} and \ref{cor:hyperbolic_inequality_d1d2_vol} hold equally well when metric volume is replaced with arithmetic measure. 
\end{rem}
%

Let $m_1\leq m_2 \leq \cdots \leq m_r$ be the exponents of the simple, simply connected, compact real-analytic Lie group of the same type as $\mathrm{G}$. These exponents can be found in \cite{Bourbaki}. Given these exponents, we define a function $f(m_1,\dots, m_r)$ by \[f(m_1,\dots, m_r)=\prod_{i=1}^r \frac{m_i!}{(2\pi)^{m_i+1}}.\]

\begin{equation} \label{exponentstable}
\begin{array}{c|c|c}
\textup{Killing-Cartan Type} & \textup{Exponents} & \textup{Lower bound for }f(m_1,\dots, m_r) \\
\hline
\rule{0pt}{2.5ex} 
A_r & 1,2,\dots, r & 10^{-32}  \\
\hline
B_r & 1,3,5,\dots,2r-1 & 10^{-16} \\
\hline
C_r & 1,3,5,\dots,2r-1 & 10^{-16} \\
\hline
D_r & 1,3,5,\dots, 2r-3, r-1 & 10^{-19} \\
\hline
E_6 & 1,4,5,7,8,11 & 10^{-15} \\
\hline 
E_7 & 1,5,7,9,11,13,17 & 10^{-13} \\
\hline
E_8 & 1,7,11,13,17,19, 23, 29 & 8434.1205\cdots \\
\hline 
F_4 & 1,5,7,11 & 10^{-9} \\
\hline
G_2 & 1,5 & 10^{-5}
\end{array}
\end{equation}

\begin{prop}\label{prop:volumedegreebound}
Let $N$ be an arithmetic simple locally symmetric space of dimension $n$, arithmetic measure less than $v$, and such that $\Lambda_N$ is contained in a principal arithmetic subgroup of $\widetilde{\mathrm{G}}(\R)$.  Let $k$ denote the field of definition of $N$.  Then the degree of $k$ is less than $c\log(v)$ where $c$ is the lower bound for $f(m_1,\dots, m_r)$ given in (\ref{exponentstable}). In particular $[k:\Q]\leq 10^{32}\log(v)$.
\end{prop}

\begin{proof}
We may assume without loss of generality that $\Lambda_N$ is a principal arithmetic subgroup of $\widetilde{\mathrm{G}}(\R)$. Prasad's formula \cite[Theorem 3.7]{P} for the covolume of a principal $S$-arithmetic subgroup implies that 
\begin{equation}\label{prasad}
v\geq \mathrm{D}_k^{\frac{1}{2}\dim \mathrm{G}}\left(\mathrm{D}_\ell /\mathrm{D}_k^{[\ell:k]}\right)^{\frac{1}{2}\mathfrak s(\mathscr G)}f(m_1,\dots, m_r)^{[k:\Q]}\mathscr E,
\end{equation}
where all notation is as in Section 3 of \cite{P}. In particular we note that $\mathrm{D}_k$ is the absolute value of the discriminant of $k$ and $\mathrm{D}_\ell$ is the absolute value of the discriminant of a certain finite degree extension $\ell$ of $k$. We have omitted the Tamagawa number $\tau_k(\widetilde{\mathrm{G}})$ in the above formula (which appears in the formula of Prasad) as the work of Kottwitz \cite{Kottwitz} shows that $\tau_k(\widetilde{\mathrm{G}})=1$ whenever $k$ is a number field. 

We now obtain a lower bound for (\ref{prasad}). First, observe that $\mathrm{D}_k\geq1$ and that $\mathrm{D}_\ell /\mathrm{D}_k^{[\ell:k]}$ is the norm from $k$ to $\Q$ of the relative discriminant of the extension $\ell/k$. Thus $\left(\mathrm{D}_\ell /\mathrm{D}_k^{[\ell:k]}\right)\geq 1$. Finally, $\mathscr E\geq 1$ by \cite[Remark 3.10]{P}. It follows that \[v\geq f(m_1,\dots, m_r)^{[k:\Q]}\] so that the proposition follows from (\ref{exponentstable}).
%
%
%
\end{proof}

The following proposition shows that when $N$ is an arithmetic hyperbolic orbifold we can obtain the same bound as in Proposition \ref{prop:volumedegreebound} without the assumption that $\Lambda_N$ is contained in a principal arithmetic group.

\begin{prop}\label{prop:degreebounds}
Let $N$ be an arithmetic hyperbolic orbifold of dimension $n\ge 4$ and hyperbolic volume less than $V$.  Let $k$ denote the field of definition of $N$.  Then the degree of $k$ is less than $c\log(V)$ where $c>0$ is an absolute, effectively computable constant. 
\end{prop}

\begin{proof}
We may assume without loss of generality that $\Lambda_N$ is a maximal arithmetic group.

We begin by considering the case in which $n=2r$ is even. In Section 3 of \cite{Bel} Belolipetsky used Prasad's volume formula to prove that 
\begin{equation}\label{eveneq1}
V \geq \frac{(2\pi)^r}{1\cdot 3 \cdots (2r-1)}\cdot \frac{4\mathrm{D}_k^{\frac{1}{2}\dim \mathrm G -1}}{10^2\left(\frac{\pi^2}{6}\right)^{[k:\Q]}}\cdot f(m_1,\dots, m_r)^{[k:\Q]}.
\end{equation}
Because $\mathrm G$ is of type $B_r$ in this case the exponents $m_1,\dots, m_r$ are equal to $1, 3, 5, \dots, 2r-1$ and $f(m_1,\dots, m_r)\geq 10^{-16}$ by (\ref{exponentstable}). As we also have $\mathrm D_k \geq 1$ and $\dim \mathrm G > 1$ we may simplify (\ref{eveneq1}) so as to obtain 
\begin{equation}\label{eveneq2}
[k:\Q]\leq c_1\log(V)
\end{equation}
where $c_1$ is a positive constant depending on $r$.

We now consider the case in which $n=2r-1$ is odd. 

When $r$ is odd Belolipetsky and Emery \cite[Section 7]{BE} have shown that 

\begin{equation}\label{oddneq1}
V\geq \frac{4\pi^r}{(r-1)!} \cdot \frac{\mathrm D_k^{r^2-r/2-2}\left(\frac{1}{2}\left(\frac{12}{\pi}\right)^2f(m_1,\dots, m_r)\right)^{[k:\Q]}}{32}.
\end{equation}

In this case $\mathrm{G}$ is of type $D_r$, hence (\ref{exponentstable}) shows that $f(m_1,\dots, m_r)\geq 10^{-19}$. As we also have $\mathrm D_k^{r^2-r/2-2}\geq 1$, we may simplify (\ref{oddneq1}) so as to obtain 

\begin{equation}\label{oddneq2}
[k:\Q]\leq c_2\log(V)
\end{equation}
where $c_2$ is a positive constant depending on $r$.

The case in which $r$ is even can be handled similarly using Section 8 of \cite{BE} and considering a variety of special cases (in this case admissible groups are of type $^{2}\mathrm D_{\frac{r-1}{2}}$ or $^{3,6}\mathrm D_4$).
\end{proof}
%
%
%





\end{document}